\documentclass[11pt,a4paper]{article}
\usepackage[latin1]{inputenc}
\usepackage{amsmath}
\usepackage{amsthm}
\usepackage{amsfonts}
\usepackage{amsfonts,amsthm,latexsym,amsmath,amssymb,amscd,epsfig,psfrag,enumerate}
\usepackage{graphics,graphicx, bezier, float, color, hyperref}
\usepackage{amssymb,url}
\usepackage{multienum}
\usepackage[table]{xcolor}
\usepackage{multicol,multirow}
\usepackage{graphicx}
\usepackage{fancyvrb}
\usepackage{parskip}
\usepackage[toc,page]{appendix}
\usepackage[top=2.7cm, bottom=2.7cm, left=1.5cm, right=1.5cm]{geometry}
\usepackage{xcolor}

\usepackage{blkarray}
\newtheorem{theorem}{Theorem}[section]
\newtheorem{lemma}{Lemma}[section]

\usepackage[none]{hyphenat}[section]
\newtheorem{definition}{Definition}[section]

\numberwithin{equation}{section}
\numberwithin{table}{section}
\numberwithin{figure}{section}

\title{$k$-Fibonacci numbers that are palindromic concatenations of two distinct Repdigits}
\author{Herbert Batte$^{1,*} $ and Florian Luca$^{1}$}
\date{}

\begin{document}
	\maketitle
	\abstract{ Let $k\ge 2$ and $\{F_n^{(k)}\}_{n\geq 2-k}$ be the sequence of $k$--generalized Fibonacci numbers whose first $k$ terms are $0,\ldots,0,0,1$ and each term afterwards is the sum of the preceding $k$ terms. In this paper, we determine all $k$-Fibonacci numbers that are palindromic concatenations of two distinct repdigits.} 
	
	
	
	\thanks{$ ^{*} $ Corresponding author}
	
\section{Introduction}\label{intro}
\subsection{Background}\label{sec:1.1}
Given an integer $k \geq 2$, the $k$-generalized Fibonacci sequence $\{F_n^{(k)}\}_{n\in\mathbb{Z}}$, is defined by the recurrence relation  
	\begin{equation*}
		F_n^{(k)} = F_{n-1}^{(k)} + F_{n-2}^{(k)} + \cdots + F_{n-k}^{(k)}, \quad \text{for all} \ n \ge 2,
	\end{equation*}
with initial values \( F_{2-k}^{(k)} = \cdots = F_{-1}^{(k)} = F_{0}^{(k)} = 0 \) and $F_1^{(k)} = 1$. This sequence generalizes the well-known Fibonacci sequence, which corresponds to the case when $k=2$. It has been investigated in numerous algebraic settings.
	
A \textit{repdigit} in base 10 is a positive integer consisting of repeated occurrences of the same digit. More formally, a number $N$ of this form can be expressed as
	\[
	N = \underbrace{\overline{d \ldots d}}_{\ell \text{ times}} = d \left( \frac{10^\ell - 1}{9} \right),
	\]
where $d$ and $\ell$ are non negative integers satisfying $0 \leq d \leq 9$ and $\ell \geq 1$. The study of Diophantine properties of recurrence sequences, particularly in relation to repdigits, has been an active area of research. In particular, previous studies have explored cases where terms in such sequences can be written as sums or concatenations of specific types of numbers. 
	
The work by Banks and Luca, \cite{banks} provided some initial insights into the number of such sequence terms, although the results were limited in scope. Additionally, the case of Fibonacci numbers formed by concatenations of two repdigits was investigated in \cite{ala}, where it was established that the largest such Fibonacci number is \( F_{14} = 377 \). This work was later extended to the sequence of $k$-generalized Fibonacci numbers in \cite{ala2}.
	
More recently, researchers have explored the connection between linear recurrence sequences and palindromic repdigits. A number is said to be a \textit{palindrome} if it reads the same backwards as forward, such as 121, 40404, and 77. The occurrence of palindromes in linear recurrence sequences has been studied extensively. For instance, Chalebgwa and Ddamulira, in \cite{chal}, showed that 151 and 616 are the only Padovan numbers that can be written as palindromic concatenations of two distinct repdigits. Similarly, it was shown in \cite{emong} that the only palindromic concatenation of two distinct repdigits appearing in the Narayana's cows sequence is 595. Further results on repdigits and palindromes in recurrence sequences can be found in \cite{batte} and \cite{kaggwa}.
	
Motivated by the work in \cite{ala2}, we instead investigate a more constrained Diophantine equation involving $k$-generalized Fibonacci numbers:
	\begin{align}\label{eq:main}
	F_n^{(k)} = \overline{\underbrace{d_1 \ldots d_1}_{\ell \text{ times}}\underbrace{d_2 \ldots d_2}_{m \text{ times}}\underbrace{d_1 \ldots d_1}_{\ell \text{ times}}},
	\end{align}
	where $ d_1, d_2 \in \{0,1,2,\dots,9\} $ with $ d_1 > 0 $, $ d_1 \neq d_2 $ and $\ell$, $m\ge 1$ because we are studying $k$-Fibonacci numbers that are palindromic concatenations of two \textit{distinct} repdigits. Also note that for $F_n^{(k)}$ to be a palindromic number of two distinct repdigits, it must have at least three digits in its decimal representation, implying that $n \ge 9$. So, we also assume $n\ge 9$ throughout the paper and prove the following main result of our study.
\subsection{Main Results}\label{sec:1.2l}
\begin{theorem}\label{thm1.1l} 
	$F_{11}^{(5)}=464$ is the only $k$-Fibonacci number that is a palindromic concatenation of two distinct repdigits.
\end{theorem}

\section{Some properties of $k$-generalized Fibonacci numbers}
The following is a well known identity:
\begin{align}\label{eq:2.1}
	F_n^{(k)} =  2^{n-2},\qquad \text{for all}\qquad 2 \le n \le k+1,
\end{align}
which can be easily proved by induction. Moreover, the first $k$-generalized Fibonacci number that is not a power of two is $F_{k+2}^{(k)} =  2^{k}-1$. The characteristic polynomial of $(F_n^{(k)})_{n\in {\mathbb Z}}$ is given by
\[
\Psi_k(x) = x^k - x^{k-1} - \cdots - x - 1.
\]
This polynomial is irreducible in $\mathbb{Q}[x]$ and it possesses a unique real root $\alpha:=\alpha(k)>1$. All the other roots of $\Psi_k(x)$ are inside the unit circle,  see \cite{bravo}. The particular root $\alpha$ can be found in the interval
\begin{align}\label{eq2.3m}
	2(1 - 2^{-k} ) < \alpha < 2,\qquad \text{for} \qquad k\ge 2,
\end{align}
as noted in \cite{bravo}. As in the classical case when $k=2$, it was shown in \cite{Brl} that 
\begin{align}\label{fn_up}
	\alpha^{n-2} \le F_n^{(k)}\le \alpha^{n-1}, \qquad \text{holds for all}\qquad n\ge1, \quad k\ge 2.
\end{align}

Lastly, for $k\ge 2$ we define
\begin{align*}
	f_k(x):=\dfrac{x-1}{2+(k+1)(x-2)}.
\end{align*}
If $2(1 - 2^{-k} ) <  x < 2$, then $\frac{d}{dx}f_k(x)<0$ and so inequality \eqref{eq2.3m} implies that
\begin{align*}
	\dfrac{1}{2}=f_k(2)<f_k(\alpha)<f_k(2(1 - 2^{-k} ))\le \dfrac{3}{4},
\end{align*}
holds for all $k\ge 3$. The inequality $1/2=f_k(2)<f_k(\alpha)<3/4$ also holds for $k=2$.

In addition, it is easy to verify that $|f_k(\alpha_i)|<1$ holds for all $2\le i\le k$, where $\alpha_i$ for $i=2,\ldots,k$ are the roots of $\Psi_k(x)$ inside the unit disk.
Lastly, it was shown in \cite{dres} that
\begin{align}\label{eq2.6m}
	F_n^{(k)}=\displaystyle\sum_{i=1}^{k}f_k(\alpha_i)\alpha_i^{n-1}~~\text{and}~~\left|F_n^{(k)}-f_k(\alpha)\alpha^{n-1}\right|<\dfrac{1}{2},
\end{align}
hold for all $k\ge 2$ and $n\ge 2-k$. The first expression in \eqref{eq2.6m} is known as the Binet-like formula for $F_n^{(k)}$. Furthermore, the second inequality expression in \eqref{eq2.6m} shows that the contribution of the zeros that are inside the unit circle to $F_n^{(k)}$ is small. However, there is a sharper estimate than \eqref{eq2.6m} and it appears in as Lemma 3 in \cite{BraLuca}. It states that if $n<2^{k/2}$, then 
\begin{align}\label{sharp}
	\left|F_n^{(k)}-2^{n-2}\right|<\dfrac{2}{2^{k/2}}.
\end{align}
\section{Methods}
\subsection{Linear forms in logarithms}
We use four times Baker-type lower bounds for nonzero linear forms in three logarithms of algebraic numbers. There are many such bounds mentioned in the literature like that of Baker and W{\"u}stholz from \cite{BW} or Matveev from \cite{matl}. Before we can formulate such inequalities we need the notion of height of an algebraic number recalled below.

\begin{definition}\label{def2.1l}
	Let $ \gamma $ be an algebraic number of degree $ d $ with minimal primitive polynomial over the integers $$ a_{0}x^{d}+a_{1}x^{d-1}+\cdots+a_{d}=a_{0}\prod_{i=1}^{d}(x-\gamma^{(i)}), $$ where the leading coefficient $ a_{0} $ is positive. Then, the logarithmic height of $ \gamma$ is given by $$ h(\gamma):= \dfrac{1}{d}\Big(\log a_{0}+\sum_{i=1}^{d}\log \max\{|\gamma^{(i)}|,1\} \Big). $$
\end{definition}
 In particular, if $ \gamma$ is a rational number represented as $\gamma:=p/q$ with coprime integers $p$ and $ q\ge 1$, then $ h(\gamma ) = \log \max\{|p|, q\} $. 
The following properties of the logarithmic height function $ h(\cdot) $ will be used in the rest of the paper without further reference:
\begin{equation}\nonumber
	\begin{aligned}
		h(\gamma_{1}\pm\gamma_{2}) &\leq h(\gamma_{1})+h(\gamma_{2})+\log 2;\\
		h(\gamma_{1}\gamma_{2}^{\pm 1} ) &\leq h(\gamma_{1})+h(\gamma_{2});\\
		h(\gamma^{s}) &= |s|h(\gamma)  \quad {\text{\rm valid for}}\quad s\in \mathbb{Z}.
	\end{aligned}
\end{equation}
With these properties, it was easily computed in Section 3 of \cite{Brl} that
\begin{align}\label{eqh}
	h\left(f_k(\alpha)\right)<3\log k, \qquad \text{for all}\qquad k\ge 2.
\end{align}

A linear form in logarithms is an expression
\begin{equation*}
	\Lambda:=b_1\log \gamma_1+\cdots+b_t\log \gamma_t,
\end{equation*}
where for us $\gamma_1,\ldots,\gamma_t$ are positive real  algebraic numbers and $b_1,\ldots,b_t$ are nonzero integers. We assume, $\Lambda\ne 0$. We need lower bounds 
for $|\Lambda|$. We write ${\mathbb K}:={\mathbb Q}(\gamma_1,\ldots,\gamma_t)$ and $D$ for the degree of ${\mathbb K}$.
We start with the general form due to Matveev, see Theorem 9.4 in \cite{matl}. 

\begin{theorem}[Matveev, see Theorem 9.4 in \cite{matl}]
	\label{thm:Matl} 
	Put $\Gamma:=\gamma_1^{b_1}\cdots \gamma_t^{b_t}-1=e^{\Lambda}-1$. Assume $\Gamma\ne 0$. Then 
	$$
	\log |\Gamma|>-1.4\cdot 30^{t+3}\cdot t^{4.5} \cdot D^2 (1+\log D)(1+\log B)A_1\cdots A_t,
	$$
	where $B\ge \max\{|b_1|,\ldots,|b_t|\}$ and $A_i\ge \max\{Dh(\gamma_i),|\log \gamma_i|,0.16\}$ for $i=1,\ldots,t$.
\end{theorem}

\subsection{Reduction methods}
Typically, the estimates from Matveev's theorem are excessively large to be practical in computations. To refine these estimates, we use a method based on the LLL--algorithm. We next explain this method.

Let $k$ be a positive integer. A subset $\mathcal{L}$ of the $k$-dimensional real vector space ${ \mathbb{R}^k}$ is called a lattice if there exists a basis $\{b_1, b_2, \ldots, b_k \}$ of $\mathbb{R}^k$ such that
\begin{align*}
	\mathcal{L} = \sum_{i=1}^{k} \mathbb{Z} b_i = \left\{ \sum_{i=1}^{k} r_i b_i \mid r_i \in \mathbb{Z} \right\}.
\end{align*}
We say that $b_1, b_2, \ldots, b_k$ form a basis for $\mathcal{L}$, or that they span $\mathcal{L}$. We
call $k$ the rank of $ \mathcal{L}$. The determinant $\text{det}(\mathcal{L})$, of $\mathcal{L}$ is defined by
\begin{align*}
	\text{det}(\mathcal{L}) = | \det(b_1, b_2, \ldots, b_k) |,
\end{align*}
with the $b_i$'s being written as column vectors. This is a positive real number that does not depend on the choice of the basis (see \cite{Cas}, Section 1.2).

Given linearly independent vectors $b_1, b_2, \ldots, b_k $ in $ \mathbb{R}^k$, we refer back to the Gram--Schmidt orthogonalization technique. This method allows us to inductively define vectors $b^*_i$ (with $1 \leq i \leq k$) and real coefficients $\mu_{i,j}$ (for $1 \leq j \leq i \leq k$). Specifically,
\begin{align*}
	b^*_i &= b_i - \sum_{j=1}^{i-1} \mu_{i,j} b^*_j,~~~
	\mu_{i,j} = \dfrac{\langle b_i, b^*_j\rangle }{\langle b^*_j, b^*_j\rangle},
\end{align*}
where \( \langle \cdot , \cdot \rangle \)  denotes the ordinary inner product on \( \mathbb{R}^k \). Notice that \( b^*_i \) is the orthogonal projection of \( b_i \) on the orthogonal complement of the span of \( b_1, \ldots, b_{i-1} \), and that \( \mathbb{R}b_i \) is orthogonal to the span of \( b^*_1, \ldots, b^*_{i-1} \) for \( 1 \leq i \leq k \). It follows that \( b^*_1, b^*_2, \ldots, b^*_k \) is an orthogonal basis of \( \mathbb{R}^k \). 
\begin{definition}
	The basis $b_1, b_2, \ldots, b_n$ for the lattice $\mathcal{L}$ is called reduced if
	\begin{align*}
		\| \mu_{i,j} \| &\leq \frac{1}{2}, \quad \text{for} \quad 1 \leq j < i \leq n,~~
		\text{and}\\
		\|b^*_{i}+\mu_{i,i-1} b^*_{i-1}\|^2 &\geq \frac{3}{4}\|b^*_{i-1}\|^2, \quad \text{for} \quad 1 < i \leq n,
	\end{align*}
	where $ \| \cdot \| $ denotes the ordinary Euclidean length. The constant $ {3}/{4}$ above is arbitrarily chosen, and may be replaced by any fixed real number $ y $ in the interval ${1}/{4} < y < 1$ {\rm(see \cite{LLL}, Section 1)}.
\end{definition}
Let $\mathcal{L}\subseteq\mathbb{R}^k$ be a $k-$dimensional lattice  with reduced basis $b_1,\ldots,b_k$ and denote by $B$ the matrix with columns $b_1,\ldots,b_k$. 
We define
\[
l\left( \mathcal{L},y\right)= \left\{ \begin{array}{c}
	\min_{x\in \mathcal{L}}||x-y|| \quad  ;~~ y\not\in \mathcal{L}\\
	\min_{0\ne x\in \mathcal{L}}||x|| \quad  ;~~ y\in \mathcal{L}
\end{array}
\right.,
\]
where $||\cdot||$ denotes the Euclidean norm on $\mathbb{R}^k$. It is well known that, by applying the
LLL--algorithm, it is possible to give in polynomial time a lower bound for $l\left( \mathcal{L},y\right)$, namely a positive constant $c_1$ such that $l\left(\mathcal{L},y\right)\ge c_1$ holds (see \cite{SMA}, Section V.4).
\begin{lemma}\label{lem2.5m}
	Let $y\in\mathbb{R}^k$ and $z=B^{-1}y$ with $z=(z_1,\ldots,z_k)^T$. Furthermore, 
	\begin{enumerate}[(i)]
		\item if $y\not \in \mathcal{L}$, let $i_0$ be the largest index such that $z_{i_0}\ne 0$ and put $\lambda:=\{z_{i_0}\}$, where $\{\cdot\}$ denotes the distance to the nearest integer.
		\item if $y\in \mathcal{L}$, put $\lambda:=1$.
	\end{enumerate}
	Then, we have 
	\[
	c_1:=\max\limits_{1\le j\le k}\left\{\dfrac{||b_1||}{||b_j^*||}\right\}\qquad\text{and}\qquad	
	\delta:=\lambda\dfrac{||b_1||}{c_1}.
	\]
\end{lemma}

In our application, we are given real numbers $\eta_0,\eta_1,\ldots,\eta_k$ which are linearly independent over $\mathbb{Q}$ and two positive constants $c_3$ and $c_4$ such that 
\begin{align}\label{2.9m}
	|\eta_0+x_1\eta_1+\cdots +x_k \eta_k|\le c_3 \exp(-c_4 H),
\end{align}
where the integers $x_i$ are bounded as $|x_i|\le X_i$ with $X_i$ given upper bounds for $1\le i\le k$. We write $X_0:=\max\limits_{1\le i\le k}\{X_i\}$. The basic idea in such a situation, from \cite{Weg}, is to approximate the linear form \eqref{2.9m} by an approximation lattice. So, we consider the lattice $\mathcal{L}$ generated by the columns of the matrix
$$ \mathcal{A}=\begin{pmatrix}
	1 & 0 &\ldots& 0 & 0 \\
	0 & 1 &\ldots& 0 & 0 \\
	\vdots & \vdots &\vdots& \vdots & \vdots \\
	0 & 0 &\ldots& 1 & 0 \\
	\lfloor C\eta_1\rfloor & \lfloor C\eta_2\rfloor&\ldots & \lfloor C\eta_{k-1}\rfloor& \lfloor C\eta_{k} \rfloor
\end{pmatrix} ,$$
where $C$ is a large constant usually of the size of about $X_0^k$ . Let us assume that we have an LLL--reduced basis $b_1,\ldots, b_k$ of $\mathcal{L}$ and that we have a lower bound $l\left(\mathcal{L},y\right)\ge c_1$ with $y:=(0,0,\ldots,-\lfloor C\eta_0\rfloor)$. Note that $ c_1$ can be computed by using the results of Lemma \ref{lem2.5m}. Then, with these notations the following result  is Lemma VI.1 in \cite{SMA}.
\begin{lemma}[Lemma VI.1 in \cite{SMA}]\label{lem2.6m}
	Let $S:=\displaystyle\sum_{i=1}^{k-1}X_i^2$ and $T:=\dfrac{1+\sum_{i=1}^{k}X_i}{2}$. If $\delta^2\ge T^2+S$, then inequality \eqref{2.9m} implies that we either have $x_1=x_2=\cdots=x_{k-1}=0$ and $x_k=-\dfrac{\lfloor C\eta_0 \rfloor}{\lfloor C\eta_k \rfloor}$, or
	\[
	H\le \dfrac{1}{c_4}\left(\log(Cc_3)-\log\left(\sqrt{\delta^2-S}-T\right)\right).
	\]
\end{lemma}
Python is used to perform all the computations in this work.

\section{Proof of Theorem \ref{thm1.1l}}
\subsection{Estimates on $n$ in terms of $\ell$ and $m$}
In Eq. \eqref{eq:main}, we see that $F_n^{(k)}$ has $2\ell+m$ digits in its palindromic representation. This means that the possible solutions to \eqref{eq:main} lie in the interval
\begin{align*}
	10^{2\ell+m-1} < F_n^{(k)}< 10^{2\ell+m}.
\end{align*}
Comparing the above inequalities with \eqref{fn_up}, we get 
\begin{align*}
	10^{2\ell+m-1} < F_n^{(k)}\le \alpha^{n-1} \qquad \text{and}\qquad \alpha^{n-2} \le F_n^{(k)}< 10^{2\ell+m}.
\end{align*}
Taking logarithms in the above inequalities, we get that the inequality
\begin{align}\label{eq2.4l}
	2\ell+m <n <5(2\ell+m)+2,
\end{align}
holds for $n\ge 9$ and $\ell$, $m\ge 1$. 
\subsection{The case $n\le k+1$}
Here, $F_n^{(k)}$ is a power of 2, so we combine \eqref{eq:main} with \eqref{eq:2.1} and write
\begin{align}\label{eq2.3l}
	2^{n-2}  &=\overline{\underbrace{d_1 \ldots d_1}_{\ell \text{ times}}\underbrace{d_2 \ldots d_2}_{m \text{ times}}\underbrace{d_1 \ldots d_1}_{\ell \text{ times}}}
	=\overline{\underbrace{d_1 \ldots d_1}_{\ell \text{ times}}}\cdot 10^{\ell+m}+\overline{\underbrace{d_2 \ldots d_2}_{m \text{ times}}}\cdot 10^{\ell}+\overline{\underbrace{d_1 \ldots d_1}_{\ell \text{ times}}}\nonumber\\
	&=d_1 \left( \frac{10^\ell - 1}{9} \right)\cdot 10^{\ell+m}+d_2 \left( \frac{10^m - 1}{9} \right)\cdot 10^{\ell}+d_1 \left( \frac{10^\ell - 1}{9} \right),
\end{align}
which similarly can be written as 
\begin{align*}
	  9\cdot 2^{n-2}=d_1 \left( 10^\ell - 1 \right)\cdot 10^{\ell+m}+d_2 \left( 10^m - 1 \right)\cdot 10^{\ell}+d_1 \left( 10^\ell - 1\right),
\end{align*}
where $ d_1, d_2 \in \{0,1,2,\dots,9\} $ with $ d_1 > 0 $, $ d_1 \neq d_2 $ and $\ell$, $m\ge 1$. Reducing the above equation modulo $10^{\ell}$, we get
\begin{align*}
	9\cdot 2^{n-2}\equiv d_1 \left( 10^\ell - 1\right) \pmod{10^\ell}.
\end{align*}
Assume for a moment that $n\ge 6$. Then $9\cdot 2^{n-2}$ and $10^\ell$ are each divisible by 16 when $\ell\ge 4$, but $d_1 \left( 10^\ell - 1\right)$ is not divisible by 16 because $\left( 10^\ell - 1\right)$ is odd and $d_1<10$. Therefore, if $n\ge 6$, then $\ell\le 3$.

Again, we rewrite \eqref{eq2.3l} as 
\begin{align*}
	9\cdot 2^{n-2}=\left(d_1 \cdot 10^{\ell}+(d_2-d_1)\right) \cdot 10^{\ell+m}+(d_1-d_2)\cdot 10^\ell - d_1,
\end{align*}
with $ d_1, d_2 \in \{0,1,2,\dots,9\} $ with $ d_1 > 0 $, $ d_1 \neq d_2 $ and $\ell$, $m\ge 1$. The integer  $|(d_1-d_2)\cdot 10^\ell - d_1|$ is at most $9\cdot 10^3+9<2^{14}$ in absolute value, so if $n\ge 16$, then $2^{14}$ divides $9\cdot 2^{n-2}$. Moreover, $\left(d_1 \cdot 10^{\ell}+(d_2-d_1)\right) \cdot 10^{\ell+m}$ is divisible by $2^{14}$ if $\ell+m\ge 14$ but $(d_1-d_2)\cdot 10^\ell - d_1$ is not divisible by $2^{14}$ for $\ell\in\{1,2,3\}$. Therefore, if $n\ge 16$, then $\ell+m\le 13$, so that $m\le 12$.

To finish off this subsection, we use Python and write a simple program to find all palindromes satisfying \eqref{eq:main} with $\ell \le 3$ and $m\le 12$. We find that none of these palindromes are powers of 2, see Appendix \ref{app0}.

\subsection{The case $n\ge k+2$}
\subsubsection{Bounding $n$ in terms of $k$}
We proceed by proving a series of results. We start with the following.
\begin{lemma}\label{lem:l}
	Let $(\ell, m, n, k)$ be a solution to the Diophantine equation \eqref{eq:main} with $n \ge 9$, $k\ge 2$ and $n\ge k+1$. Then
	$$\ell<4\cdot 10^{12}k^4 (\log k)^2\log n.$$
\end{lemma}
\begin{proof}
Suppose that $n$, $k$ are positive integer solutions to \eqref{eq:main} with $n\ge k+2$ and $k\ge 2$. We rewrite \eqref{eq2.6m} using \eqref{eq2.3l} as
\begin{align*}
	\left|\dfrac{1}{9}\left(d_1\cdot 10^{2\ell+m}-(d_1-d_2)\cdot 10^{\ell+m} +(d_1-d_2)\cdot 10^{\ell}-d_1 \right)-f_k(\alpha)\alpha^{n-1}\right|<\dfrac{1}{2}.
\end{align*}
This tells us that for $\ell$, $m\ge 1$,
\begin{align*}
	\left|f_k(\alpha)\alpha^{n-1}-\dfrac{1}{9}\left(d_1\cdot 10^{2\ell+m}\right)\right|
	&<\dfrac{1}{2}+\dfrac{1}{9}\left(|d_1-d_2|\cdot 10^{\ell+m} +|d_1-d_2|\cdot 10^{\ell}+d_1 \right)\\
	&\le \dfrac{1}{2}+ 10^{\ell+m} +10^{\ell}+1 \\
	&=10^{\ell+m}\left(\dfrac{3}{2\cdot 10^{\ell+m}}+ 1 +\dfrac{1}{ 10^{m}} \right)\\
	&\le 10^{\ell+m}\left(\dfrac{3}{2\cdot 10^{2}}+ 1 +\dfrac{1}{ 10^{1}} \right)\\
	&<1.2\cdot 10^{\ell+m}.
\end{align*}

Dividing through both sides of the above inequality by $(d_1\cdot 10^{2\ell+m})/9$, we get
\begin{align}\label{gam1}
	\left|\dfrac{9}{d_1}\cdot 10^{-2\ell-m}f_k(\alpha)\alpha^{n-1}-1\right|
	&<11\cdot 10^{-\ell}.
\end{align}
Let 
$$
\Gamma_1=\dfrac{9}{d_1}\cdot 10^{-2\ell-m}f_k(\alpha)\alpha^{n-1}-1.
$$
Notice that $\Gamma_1\ne 0$, otherwise we would have
\[
d_1\cdot 10^{2\ell + m} = 9 f_k(\alpha) \alpha^{n-1}.
\]
Applying an automorphism from the Galois group of the decomposition field of \(\Psi_k(x)\) over \(\mathbb{Q}\), which maps \(\alpha\) to its conjugate \(\alpha_j\) for \(2 \leq j \leq k\), and then considering the absolute value, we derive the following for any \(j \geq 2\):
\[
10^3 \le  d_1\cdot  10^{2\ell + m} = \left| 9 f_k(\alpha_j) \alpha_j^{n-1} \right| < 9,
\]
which is not true. The algebraic number field containing the following $\gamma_i$'s is $\mathbb{K} := \mathbb{Q}(\alpha)$. We have $D = k$, $t :=3$,
\begin{equation}\nonumber
	\begin{aligned}
		\gamma_{1}&:=9f_k(\alpha)/d_1,\quad\gamma_{2}:=\alpha,\qquad~~\gamma_{3}:=10,\\
		b_{1}&:=1,\qquad \qquad~~~ b_{2}:=n-1,\quad b_{3}:=-2\ell-m.
	\end{aligned}
\end{equation}
Since 
\begin{align*}
h(\gamma_{1})=h(9f_k(\alpha)/d_1)< \log 9+3\log k+\log 9 <10\log k,
\end{align*}
by \eqref{eqh}, $h(\gamma_{2})=h(\alpha)<(\log \alpha)/k <0.7/k$ and $h(\gamma_{3})=h(10)= \log 10 $, we take $A_1:=10k\log k$, $A_2:=0.7$ and $A_3:=k\log 10$. Next, $B \geq \max\{|b_i|:i=1,2,3\}$. By \eqref{eq2.4l}, $2\ell+m<n$, so we take $B:=n$. Now, by Theorem \ref{thm:Matl},
\begin{align}\label{gam2}
	\log |\Gamma_1| &> -1.4\cdot 30^{6} \cdot 3^{4.5}\cdot k^2 (1+\log k)(1+\log n)\cdot 10k\log k\cdot 0.7\cdot k\log 10\nonumber\\
	&> -8.4\cdot 10^{12}k^4 (\log k)^2\log n.
\end{align}
for $k\ge 2$ and $n\ge 9$. Comparing \eqref{gam1} and \eqref{gam2}, we get
\begin{align*}
	\ell\log 10-\log 11 &<8.4\cdot 10^{12}k^4 (\log k)^2\log n,
\end{align*}
which leads to $\ell<4\cdot 10^{12}k^4 (\log k)^2\log n$.
\end{proof}

Next, we prove the following.
\begin{lemma}\label{lem:m}
	Let $(\ell, m, n, k)$ be a solution to the Diophantine equation \eqref{eq:main} with $n \ge 9$, $k\ge 2$ and $n\ge k+1$. Then
	$$m<3.5\cdot 10^{24}k^8 (\log k)^3(\log n)^2\qquad\text{and}\qquad	n	<3\cdot 10^{31}k^8 (\log k)^5.$$
\end{lemma}
\begin{proof}
	As before, we again rewrite \eqref{eq2.6m} using \eqref{eq2.3l} as
	\begin{align*}
		\left|\dfrac{1}{9}\left(d_1\cdot 10^{2\ell+m}-(d_1-d_2)\cdot 10^{\ell+m} +(d_1-d_2)\cdot 10^{\ell}-d_1 \right)-f_k(\alpha)\alpha^{n-1}\right|<\dfrac{1}{2}.
	\end{align*}
	This tells us that
	\begin{align*}
		\left|f_k(\alpha)\alpha^{n-1}-\dfrac{1}{9}\left(d_1\cdot 10^{2\ell+m}-(d_1-d_2)\cdot 10^{\ell+m}\right)\right|
		&<\dfrac{1}{2}+\dfrac{1}{9}\left(|d_1-d_2|\cdot 10^{\ell}+d_1 \right)\\
		&<\dfrac{1}{2}+ 10^{\ell}+1
		=10^{\ell}\left(\dfrac{3}{2\cdot 10^{\ell}}+ 1 \right)\\
		&<10^{\ell}\left(\dfrac{3}{2\cdot 10^{1}}+ 1  \right)\\
		&<1.2\cdot 10^{\ell},
	\end{align*}
	for $\ell\ge 1$. 
	
Dividing through both sides of the above inequality by $(d_1\cdot 10^{2\ell+m}-(d_1-d_2)\cdot 10^{\ell+m})/9$, we get
	\begin{align}\label{gam3}
		\left|\dfrac{9 }{(d_1\cdot 10^{\ell}-(d_1-d_2))}\cdot 10^{-\ell-m}f_k(\alpha)\alpha^{n-1}-1\right|
		&<11\cdot 10^{-m}.
	\end{align}
	Let 
	$$
	\Gamma_2=\dfrac{9 }{(d_1\cdot 10^{\ell}-(d_1-d_2))}\cdot 10^{-\ell-m}f_k(\alpha)\alpha^{n-1}-1.
	$$
Again, $\Gamma_2\ne 0$, otherwise we would have
\[
	d_1\cdot 10^{2\ell+m}-(d_1-d_2)\cdot 10^{\ell+m} = 9 f_k(\alpha) \alpha^{n-1}.
\]
By a similar argument as before, if we apply an automorphism from the Galois group of the decomposition field of \(\Psi_k(x)\) over \(\mathbb{Q}\), which maps \(\alpha\) to its conjugate \(\alpha_j\) for \(2 \leq j \leq k\), and then considering the absolute value, we have that for any \(j \geq 2\),
	\[
	10^2 \le  d_1\cdot 10^{2\ell+m}-(d_1-d_2)\cdot 10^{\ell+m} = \left| 9 f_k(\alpha_j) \alpha_j^{n-1} \right| < 9,
	\]
	which is false. The algebraic number field containing the following $\gamma_i$'s is $\mathbb{K} := \mathbb{Q}(\alpha)$. We have $D = k$, $t :=3$,
	\begin{equation}\nonumber
		\begin{aligned}
			\gamma_{1}&:=9f_k(\alpha)/(d_1\cdot 10^{\ell}-(d_1-d_2)),\quad\gamma_{2}:=\alpha,\qquad~~\gamma_{3}:=10,\\
			b_{1}&:=1,\qquad \qquad\qquad\qquad\qquad\qquad~~~ b_{2}:=n-1,\quad b_{3}:=-\ell-m.
		\end{aligned}
	\end{equation}
	Since 
	\begin{align*}
		h(\gamma_{1})&=h(9f_k(\alpha)/(d_1\cdot 10^{\ell}-(d_1-d_2))
		\le h(9)+h(f_k(\alpha))+h(d_1\cdot 10^{\ell}-(d_1-d_2))\\
		&< \log 9+3\log k+\log 9+\ell\log 10+\log 9+\log 2 \\
		&< \log 9+3\log k+\log 9+\left(4\cdot 10^{12}k^4 (\log k)^2\log n\right)\log 10+\log 9+\log 2 \\
		&=k^4 (\log k)^2\log n\left(\dfrac{4\log 9}{k^4 (\log k)^2\log n}+\dfrac{3}{k^4 \log k\log n}+4\cdot 10^{12}\log 10\right)\\
		&<9.22\cdot 10^{12}k^4 (\log k)^2\log n,
	\end{align*}
	by \eqref{eqh}, we can take $A_1:=9.22\cdot 10^{12}k^5 (\log k)^2\log n$. As before, $A_2:=0.7$, $A_3:=k\log 10$, and $B:=n$. 
	
	Now, by Theorem \ref{thm:Matl}, we have that for $k\ge 2$ and $n\ge 9$
	\begin{align}\label{gam4}
		\log |\Gamma_2| &> -1.4\cdot 30^{6} \cdot 3^{4.5}\cdot k^2 (1+\log k)(1+\log n)\cdot 9.22\cdot 10^{12}k^5 (\log k)^2\log n\cdot 0.7\cdot k\log 10\nonumber\\
		&> -7.9\cdot 10^{24}k^8 (\log k)^3(\log n)^2.
	\end{align}
	 Comparing \eqref{gam3} and \eqref{gam4}, we get
	\begin{align*}
		m\log 10-\log 11 &<7.9\cdot 10^{24}k^8 (\log k)^3(\log n)^2,
	\end{align*}
	which leads to 
	$$m<3.5\cdot 10^{24}k^8 (\log k)^3(\log n)^2.$$

At this point, we go back to \eqref{eq2.4l} and use the inequality above and the bound in Lemma \ref{lem:l}, that is
\begin{align*}
	n&<	5(2\ell+m)+2\\
	&<5\left(2\cdot 4\cdot 10^{12}k^4 (\log k)^2\log n+3.5\cdot 10^{24}k^8 (\log k)^3(\log n)^2\right)+2\\
	&=21\cdot 10^{24}k^8 (\log k)^3(\log n)^2\left(\dfrac{8\cdot 10^{12}k^4 (\log k)^2\log n}{ 3.5\cdot 10^{24}k^8 (\log k)^3(\log n)^2}+1+\dfrac{2/5}{3.5\cdot 10^{24}k^8 (\log k)^3(\log n)^2} \right)                  \\
	&<2.2\cdot 10^{25}k^8 (\log k)^3(\log n)^2.
\end{align*}
We get
\begin{align}\label{boundn}
	n	&<3\cdot 10^{31}k^8 (\log k)^5.
\end{align}
This completes the proof of Lemma \ref{lem:m}
\end{proof}
We proceed by distinguishing between two cases on $k$.

\subsubsection{Case I: $k \leq 900$} 
Assuming that $k \leq 900$, it follows from \eqref{boundn} that  
\[
n < 3 \cdot 10^{31} k^8 (\log k)^5 < 1.9 \cdot 10^{59}.
\]
Our objective is to derive a tighter upper bound for $n$. To accomplish this, we revisit \eqref{gam1} and recall the previously obtained result:
\[
\Gamma_1 := \frac{9}{d_1} \cdot 10^{-2\ell - m} f_k(\alpha) \alpha^{n-1} - 1 = e^{\Lambda_1} - 1.
\]
Since we established that $\Gamma_1 \neq 0$, it follows that $\Lambda_1 \neq 0$. Now, assuming $\ell \geq 2$, we obtain the inequality  
\[
\left| e^{\Lambda_1} - 1 \right| = |\Gamma_1| < 0.5,
\]
which leads to $e^{|\Lambda_1|} \leq 1 + |\Gamma_1| < 1.5$. Consequently, we derive  
\[
\left| (n - 1) \log \alpha - (2\ell + m) \log 10 + \log \left( \frac{9 f_k(\alpha)}{d_1} \right) \right| < \frac{16.5}{10^\ell}.
\]
Now, we apply the LLL-algorithm for each $k \in [2,900]$ and $d_1 \in \{1, \dots, 9\}$ to obtain a lower bound for the smallest nonzero value of the above linear form, constrained by integer coefficients with absolute values not exceeding $n < 1.9 \cdot 10^{59}$. Specifically, we consider the lattice  
\[
\mathcal{A} = \begin{pmatrix} 
	1 & 0 & 0 \\ 
	0 & 1 & 0 \\ 
	\lfloor C\log \alpha\rfloor & \lfloor C\log (1/10)\rfloor & \lfloor C\log \left(9 f_k(\alpha)/d_1\right) \rfloor
\end{pmatrix},
\]
where we set $C := 2.1\cdot 10^{178}$ and $y := (0,0,0)$. Applying Lemma \ref{lem2.5m}, we obtain 
\[
l(\mathcal{L},y) = |\Lambda| > c_1 = 10^{-63} \quad \text{and} \quad \delta = 10^{61}. 
\]
Using Lemma \ref{lem2.6m}, we conclude that $S =1.1 \cdot 10^{119}$ and $T = 5.8 \cdot 10^{59}$. Given that $\delta^2 \geq T^2 + S$, and selecting $c_3 := 16.5$ and $c_4 := \log 10$, we establish the bound $\ell \leq 118$.

Next, we revisit \eqref{gam3}, that is,
$$\Gamma_2:=\dfrac{9 }{(d_1\cdot 10^{\ell}-(d_1-d_2))}\cdot 10^{-\ell-m}f_k(\alpha)\alpha^{n-1}-1=e^{\Lambda_2}-1.$$
We already showed that $\Gamma_2\ne 0$, so $\Lambda_2\ne 0$. Assume for a moment that $m\ge 2$. As before, this implies that 
\begin{align*}
	\left|(n - 1) \log \alpha - (\ell + m) \log 10 + \log \left( \dfrac{9 f_k(\alpha)}{d_1\cdot 10^{\ell}-(d_1-d_2)} \right)\right|<\dfrac{16.5}{10^m}.
\end{align*}
Now, for each $k \in [2, 900]$, $\ell\in[1,118]$, $d_1, d_2\in \{0,\ldots,9\}$, $d_1>0$ and $d_1\ne d_2$, we use the LLL--algorithm to compute a lower bound for the smallest nonzero number of the linear form above with integer coefficients not exceeding $n<1.9\cdot 10^{59}$ in absolute value. So, we consider the approximation lattice
$$ \mathcal{A}=\begin{pmatrix}
	1 & 0 & 0 \\
	0 & 1 & 0 \\
	\lfloor C\log \alpha\rfloor & \lfloor C\log (1/10)\rfloor& \lfloor C\log \left(9 f_k(\alpha)/(d_1\cdot 10^{\ell}-(d_1-d_2))\right)\rfloor
\end{pmatrix} ,$$
with $C:=  10^{179}$ and choose $y:=\left(0,0,0\right)$. Now, by Lemma \ref{lem2.5m}, we get $$l\left(\mathcal{L},y\right)=|\Lambda|>c_1=10^{-62}\qquad\text{and}\qquad\delta=10^{60}.$$
So, Lemma \ref{lem2.6m} gives $S =1.1 \cdot 10^{119}$ and $T = 5.8 \cdot 10^{59}$ as before. So choosing $c_3:=16.5$ and $c_4:=\log 10$, we get $m \le 120$.

Therefore, we use \eqref{eq2.4l} to get $n\le 2138$. Finally, we use Python to find all $F_n^{(k)}$ values in the ranges $k\in[2,900]$ and $n\in[8,2138]$ that are palindromic concatenation of two distinct repdigits. We find only one value given in Theorem \ref{thm1.1l}, see Appendix \ref{app1}.

\subsubsection{Case II: $k > 900$}  
If $k > 900$, then 
\begin{align}\label{boundn2}
	n	&<3\cdot 10^{31}k^8 (\log k)^5<2^{k/2}. 
\end{align}
We prove the following result.
\begin{lemma}\label{lem:nm}
	Let $(\ell, m, n, k)$ be a solution to the Diophantine equation \eqref{eq:main} with $n \ge 9$, $k>900$ and $n\ge k+1$. Then
	$$k<3.2\cdot 10^{31} \qquad\text{and}\qquad n<6.7\cdot 10^{292}.$$
\end{lemma}
\begin{proof}
Since $k>900$, then \eqref{boundn2} holds and we can use \eqref{sharp}. So, we combine \eqref{sharp} with \eqref{eq2.3l} and write
\begin{align*}
	\left|\dfrac{1}{9}\left(d_1\cdot 10^{2\ell+m}-(d_1-d_2)\cdot 10^{\ell+m} +(d_1-d_2)\cdot 10^{\ell}-d_1 \right)-2^{n-2}\right|<\dfrac{2}{2^{k/2}}.
\end{align*}
Therefore
\begin{align*}
	\left|\dfrac{1}{9}\cdot d_1\cdot 10^{2\ell+m}-2^{n-2}\right|&<\dfrac{2}{2^{k/2}}+\dfrac{1}{9}\left(|d_1-d_2|\cdot 10^{\ell+m} +|d_1-d_2|\cdot 10^{\ell}+d_1 \right)\\
	&<\dfrac{2}{2^{k/2}}+1.2\cdot 10^{\ell+m},
\end{align*}
and dividing both sides by $2^{n-2}$ gives
\begin{align*}
	\left|\dfrac{1}{9}\cdot d_1\cdot 10^{2\ell+m}\cdot 2^{-(n-2)}-1\right| &<\dfrac{1}{2^{k/2}}+1.2\cdot \dfrac{10^{\ell+m}}{2^{n-2}}\\
	&<\dfrac{1}{2^{k/2}}+1.2\cdot \dfrac{21}{10^\ell},
\end{align*}
via the inequality in \eqref{eq2.4l}. We have also used the fact that $2/(2^{k/2}\cdot 2^{n-2})<1/2^{k/2}$, i.e. $2/2^{n-2}<1$, for $n\ge 9$. Therefore,
\begin{align}\label{gam5}
	\left|\dfrac{1}{9}\cdot d_1\cdot 10^{2\ell+m}\cdot 2^{-(n-2)}-1\right| 	&<\dfrac{27}{2^{\min\{k/2,\, \ell\log_2 10\}}}.
\end{align}

Let 
$$
\Gamma_3=\dfrac{1}{9}\cdot d_1\cdot 10^{2\ell+m}\cdot 2^{-(n-2)}-1.
$$
Notice that $\Gamma_3\ne 0$, otherwise we would have
\[
d_1\cdot 10^{2\ell + m} = 9\cdot 2^{n-2},
\]
which is not true because the left-hand side is divisible by 5, while the right-hand side is not. The algebraic number field containing the following $\gamma_i$'s is $\mathbb{K} := \mathbb{Q}$. We have $D = 1$, $t :=3$,
\begin{equation}\nonumber
	\begin{aligned}
		\gamma_{1}&:=d_1/9,\quad\gamma_{2}:=2,\qquad\qquad~\gamma_{3}:=10,\\
		b_{1}&:=1,\qquad ~~ b_{2}:=-(n-2),\quad b_{3}:=2\ell+m.
	\end{aligned}
\end{equation}

Since $h(\gamma_{1})\le h(9)+h(d_1)\le 2\log 9$,
$h(\gamma_{2})=h(2)=\log 2$ and $h(\gamma_{3})=h(10)= \log 10 $, we take $A_1:=2\log 9$, $A_2:=\log 2$ and $A_3:=\log 10$. Next, $B :=n\geq \max\{|b_i|:i=1,2,3\}$ and by Theorem \ref{thm:Matl}, we have
\begin{align}\label{gam6}
	\log |\Gamma_3| &> -1.4\cdot 30^{6} \cdot 3^{4.5}\cdot 1^2 (1+\log 1)(1+\log n)\cdot 2\log 9\cdot \log 2\cdot \log 10\nonumber\\
	&> -1.5\cdot 10^{12}\log n,
\end{align}
for $n\ge 9$. Comparing \eqref{gam5} and \eqref{gam6}, we get
\begin{align*}
	\min\{k/2,\, \ell\log_2 10\}\log 2-\log 27 &<1.5\cdot 10^{12}\log n.
\end{align*}
Two cases arise here.
\begin{enumerate}[(a)]
	\item If $\min\{k/2,\, \ell\log_2 10\}:=k/2$, then $(k/2)\log 2-\log 27 <1.5\cdot 10^{12}\log n$, so that
	\begin{align*}
		k &<4.4\cdot 10^{12}\log n.
	\end{align*}
	By \eqref{boundn2}, we have 
	$$n	<3\cdot 10^{31}k^8 (\log k)^5,$$
	 and so 
	 $$\log n	<73+13\log k<117\log k,$$
	  for $k> 900$. Therefore, we have 
	  $$k<5.3\cdot 10^{14}\log k,$$
	   which gives $k<8.2\cdot 10^{15}$.
	
	\item If $\min\{k/2,\, \ell\log_2 10\}:=\ell\log_2 10$, then $(\ell\log_2 10)\log 2-\log 27 <1.5\cdot 10^{12}\log n$, so that
	\begin{align*}
		\ell &<10^{12}\log n.
	\end{align*}
		We proceed as in \eqref{gam5} via
	\begin{align*}
		\left|\dfrac{1}{9}\left(d_1\cdot 10^{2\ell+m}-(d_1-d_2)\cdot 10^{\ell+m} +(d_1-d_2)\cdot 10^{\ell}-d_1 \right)-2^{n-2}\right|<\dfrac{2}{2^{k/2}}.
	\end{align*}
This means that
	\begin{align*}
		\left|\dfrac{1}{9}\left(d_1\cdot 10^{2\ell+m}-(d_1-d_2)\cdot 10^{\ell+m}\right)-2^{n-2}\right|
		&<\dfrac{2}{2^{k/2}}+\dfrac{1}{9}\left(|d_1-d_2|\cdot 10^{\ell}+d_1 \right)\\
		&<\dfrac{2}{2^{k/2}}+1.2\cdot 10^{\ell}.
	\end{align*}
Dividing through both sides of the above inequality by $2^{n-2}$, we get
	\begin{align*}
		\left|\dfrac{1}{9}\left(d_1\cdot 10^{2\ell+m}-(d_1-d_2)\cdot 10^{\ell+m}\right)\cdot 2^{-(n-2)}-1\right|
		&<\dfrac{1}{2^{k/2}}+1.2\cdot \dfrac{10^\ell}{2^{n-2}}.
	\end{align*}
Therefore, since $n\ge k+2$ (hence $n-2>k/2$) and $k/2> \ell\log_2 10$, we have 
\begin{align}\label{gam7}
	\left|\dfrac{1}{9}\left(d_1\cdot 10^{2\ell+m}-(d_1-d_2)\cdot 10^{\ell+m}\right)\cdot 2^{-(n-2)}-1\right|
	&<\dfrac{8}{2^{k/2}}.
\end{align}
	Let 
	$$
	\Gamma_4=\dfrac{1}{9}\left(d_1\cdot 10^{\ell}-(d_1-d_2)\right)\cdot 10^{\ell+m}\cdot 2^{-(n-2)}-1.
	$$
	Again, $\Gamma_4\ne 0$, otherwise we would have
	\[
	d_1\cdot 10^{2\ell+m}-(d_1-d_2)\cdot 10^{\ell+m} = 9 \cdot 2^{n-2},
	\]
which is false because the left-hand side is divisible by 5, while the right-hand side is not. The algebraic number field containing the following $\gamma_i$'s is $\mathbb{K} := \mathbb{Q}$. We have $D = 1$, $t :=3$,
	\begin{equation}\nonumber
		\begin{aligned}
			\gamma_{1}&:=(d_1\cdot 10^{\ell}-(d_1-d_2))/9,\quad\gamma_{2}:=2,\qquad\qquad~\gamma_{3}:=10,\\
			b_{1}&:=1,\qquad \qquad\qquad\qquad\qquad~~ b_{2}:=-(n-2),\quad b_{3}:=\ell+m.
		\end{aligned}
	\end{equation}
	Since 
	\begin{align*}
		h(\gamma_{1})&=h((d_1\cdot 10^{\ell}-(d_1-d_2))/9)\\
		&< \log 9+\ell\log 10+3\log 9+4\log 2 \\
		&< \log 9+10^{12}\log n\log 10+3\log 9+4\log 2 \\
		&<2.4\cdot 10^{12}\log n,
	\end{align*}
we can take $A_1:=2.4\cdot 10^{12}\log n$. As before, $A_2:=\log 2$, $A_3:=\log 10$, and $B:=n$. Now, by Theorem \ref{thm:Matl},
	\begin{align}\label{gam8}
		\log |\Gamma_4| &> -1.4\cdot 30^{6} \cdot 3^{4.5}\cdot 1^2 (1+\log )(1+\log n)\cdot 2.4\cdot 10^{12}\log n\cdot \log 2\cdot \log 10\nonumber\\
		&> -8\cdot 10^{23}(\log n)^2.
	\end{align}
	for $k\ge 2$ and $n\ge 9$. Comparing \eqref{gam7} and \eqref{gam8}, we get
	\begin{align*}
		(k/2)\log 2-\log 8 &<8\cdot 10^{23}(\log n)^2,
	\end{align*}
	which leads to $k<2.4\cdot 10^{24}(\log n)^2$.
By \eqref{boundn2}, 
$$n	<3\cdot 10^{31}k^8 (\log k)^5,$$ 
and so $\log n<117\log k$, for $k> 900$. Therefore, we have 
$$k<3.3\cdot 10^{28}(\log k)^2,$$ which gives 
$k<3.2\cdot 10^{31},$ 
and thus $n<6.7\cdot 10^{292}$.
\end{enumerate}
This completes the proof.
\end{proof}

We reduce the bounds in Lemma \ref{lem:nm}. To do this, we revisit \eqref{gam5} and recall that
\[
\Gamma_3 := \dfrac{1}{9}\cdot d_1\cdot 10^{2\ell+m}\cdot 2^{-(n-2)}-1.
\]
Since we showed that $\Gamma_3 \neq 0$, it follows that $\Lambda_3 \neq 0$. Since $\min\{k/2,\, \ell\log_2 10\}\ge 6$, we obtain the inequality  
\[
\left| e^{\Lambda_3} - 1 \right| = |\Gamma_3| < 0.5,
\]
which leads to $e^{|\Lambda_3|} \leq 1 + |\Gamma_3| < 1.5$. Therefore 
\[
\left|  \log \left(\dfrac{d_1}{9}\right) + (2\ell + m) \log 10 -(n-2) \log 2  \right| < \dfrac{40.5}{2^{\min\{k/2,\, \ell\log_2 10\}}}.
\]
For each $d_1\in\{1,\ldots , 9\}$, we apply the LLL-algorithm to obtain a lower bound for the smallest nonzero value of the above linear form, bounded by integer coefficients with absolute value less than $n < 6.7 \cdot 10^{292}$. Here, we consider the lattice  
\[
\mathcal{A}' = \begin{pmatrix} 
	1 & 0 & 0 \\ 
	0 & 1 & 0 \\ 
	\lfloor C\log (d_1/9)\rfloor & \lfloor C\log 10\rfloor & \lfloor C\log (1/2) \rfloor
\end{pmatrix},
\]
where we set $C := 10^{879}$ and $y := (0,0,0)$. Applying Lemma \ref{lem2.5m}, we obtain  
\[
l(\mathcal{L},y) = |\Lambda| > c_1 = 10^{-297} \quad \text{and} \quad \delta = 1.61\cdot 10^{295}.
\]
Using Lemma \ref{lem2.6m}, we conclude that $S = 1.4 \cdot 10^{586}$ and $T = 1.1 \cdot 10^{293}$. Given that $\delta^2 \geq T^2 + S$, and choosing $c_3 := 40.5$ and $c_4 := \log 2$, we get  $\min\{k/2,\, \ell\log_2 10\} \leq 1944$. 

We proceed the analysis in two ways.
\begin{enumerate}[(a)]
	\item If $\min\{k/2,\, \ell\log_2 10\}:=k/2$, then $k/2 \leq 1944$, or $k\le 3888$.

	\item If $\min\{k/2,\, \ell\log_2 10\}:=\ell\log_2 10$, then $\ell\log_2 10\le 1944$. Therefore $\ell \le 585$.
	We revisit equation \eqref{gam7} and recall the expression
	\[
\Gamma_4:=\dfrac{1}{9}\left(d_1\cdot 10^{\ell}-(d_1-d_2)\right)\cdot 10^{\ell+m}\cdot 2^{-(n-2)}-1.
	\]
	Since we have already shown that \(\Gamma_4 \neq 0\), it follows that \(\Lambda_4 \neq 0\). Since \(k >900\), we obtain the inequality
	\[
	\left| e^{\Lambda_4} - 1 \right| = |\Gamma_4| < 0.5,
	\]
	which implies that \(e^{|\Lambda_4|} \leq 1 + |\Gamma_4| < 1.5\). Consequently, we arrive at
	\[
	\left|  \log \left(\frac{d_1\cdot 10^{\ell}-(d_1-d_2)}{9}\right) + (\ell + m) \log 10 -(n-2) \log 2  \right| < \frac{12}{2^{k/2}}.
	\]
We use the LLL-algorithm, restricting integer coefficients to absolute values below \(n < 6.7 \times 10^{292}\). For this purpose, we consider the lattice
	\[
	\mathcal{A}'' = \begin{pmatrix} 
		1 & 0 & 0 \\ 
		0 & 1 & 0 \\ 
		\lfloor C\log ((d_1\cdot 10^{\ell}-(d_1-d_2))/9)\rfloor & \lfloor C\log 10\rfloor & \lfloor C\log (1/2) \rfloor
	\end{pmatrix},
	\]
	where we define \(C := 10^{879}\) and set \(y := (0,0,0)\) as before. We use the same values for $\delta$, $S$ and $T$ as before, and by selecting \(c_3 := 12\) and \(c_4 := \log 2\), it follows that $k/2 \leq 1942$ from which we get $k\le 3884$.
	
\end{enumerate}
Therefore, we always have $k\le 3888$. 

This bound on $k$ gives $n<10^{65}$ via inequality \eqref{boundn2}, so we repeat the reduction again with this new upper bound on $n$. We reconsider equation \eqref{gam5} again as before and apply the LLL-algorithm to determine a lower bound for the smallest nonzero value of this linear form, where the integer coefficients are bounded by $n <  10^{65}$. We use the same lattice  $\mathcal{A}'$ as before but now $C := 10^{195}$ and $y := (0,0,0)$. By applying Lemma \ref{lem2.5m}, we establish  
\[
l(\mathcal{L},y) = |\Lambda| > c_1 = 10^{-68} \quad \text{and} \quad \delta = 7.3\cdot 10^{66}.
\]
Using Lemma \ref{lem2.6m}, we conclude that $S = 3 \cdot 10^{130}$ and $T = 1.51 \cdot 10^{65}$. Since $\delta^2 \geq T^2 + S$, and choosing $c_3 := 40.5$ and $c_4 := \log 2$, it follows that $\min\{k/2,\, \ell\log_2 10\} \leq 434$.  

We again analyze two cases:
\begin{enumerate}[(a)]
	\item If $\min\{k/2,\, \ell\log_2 10\} := k/2$, then $k/2 \leq 434$, which implies $k \leq 868$. This contradicts the working assumption that $k>900$.
	
	\item If $\min\{k/2,\, \ell\log_2 10\} := \ell\log_2 10$, then $\ell\log_2 10 \leq 434$, giving $\ell \leq 130$.
	
	Revisiting equation \eqref{gam7}, we again apply the LLL-algorithm, with integer coefficients restricted to absolute values below $n <  10^{65}$. By the same lattice $\mathcal{A}''$ as before, we have $C := 10^{195}$ and $y := (0,0,0)$. Using Lemma \ref{lem2.5m} and Lemma \ref{lem2.6m}, we have $\delta=7.3\cdot 10^{66}$, $S = 3 \cdot 10^{130}$ and $T = 1.51 \cdot 10^{65}$. Selecting $c_3 := 12$ and $c_4 := \log 2$, we deduce that $k/2 \leq 428$, leading to $k \leq 856$. This also contradicts $k>900$.
	
\end{enumerate}
This completes the proof to Theorem \ref{thm1.1l}. \qed

\section*{Acknowledgments} 
The first author thanks the Mathematics division of Stellenbosch University for funding his PhD studies.

\section*{Addresses}

$ ^{1} $ Mathematics Division, Stellenbosch University, Stellenbosch, South Africa.

Email: \url{hbatte91@gmail.com}

Email: \url{fluca@sun.ac.za}

\appendix
\section{Appendices}
\subsection{Py Code I}\label{app0}
\begin{verbatim}
def is_power_of_two(n):
    """Returns True if n is a power of 2"""
    return n > 0 and (n & (n - 1)) == 0
def find_power_of_two_palindromes(max_l=3, max_m=12):
    valid_palindromes = []
    for d1 in range(1, 10):  # d1 > 0 (no leading zero)
        for d2 in range(10):
            if d1 == d2:
               continue
            for l in range(1, max_l + 1):
               for m in range(1, max_m + 1):
                   left = str(d1) * l
                   middle = str(d2) * m
                   right = str(d1) * l
                   pal_str = left + middle + right
                   pal_num = int(pal_str)
                   if is_power_of_two(pal_num):
                      valid_palindromes.append((pal_num, d1, d2, l, m))
    return valid_palindromes
# Run it
results = find_power_of_two_palindromes()
# Print result
print(f"Found {len(results)} palindromes that are powers of 2:")
for val, d1, d2, l, m in results:
    print(f"{val} = 2^{val.bit_length() - 1}, from d1={d1}, d2={d2}, l={l}, m={m}")
\end{verbatim}

\subsection{Py Code II}\label{app1}
\begin{verbatim}
def k_fibonacci_sequence(k, n_max):
    """Generates the k-Fibonacci sequence up to n_max terms."""
    F = [0] * (k - 1) + [1]
    for _ in range(n_max - len(F)):
       F.append(sum(F[-k:]))
    return F
def is_palindromic_concatenation(num):
    s = str(num)
    n = len(s)
    for l in range(1, n // 2 + 1):
        m = n - 2 * l
        if m <= 0:
           continue
        part1 = s[:l]
        part2 = s[l:l + m]
        part3 = s[l + m:]
        if (
           part1 == part3
           and len(set(part1)) == 1  # all digits same
           and len(set(part2)) == 1
           and part1 != part2
        ):
           return True
    return False
def find_all_palindromic_k_fibs(k_min=2, k_max=900, n_max=2138):
    results = []
    for k in range(k_min, k_max + 1):
        fibs = k_fibonacci_sequence(k, n_max + 1)
        for i, val in enumerate(fibs:)
            math_n i+2-k
            if val > 0 and is_palindromic_concatenation(val):
                 results.append((k, n, val))
        return results
# Run the full check
results = find_all_palindromic_k_fibs()
# Print results
print("Found palindromic repdigit k-Fibonacci numbers in 2 <= k<=900 ,n <= 2138:")
for k, n, value in results:
    print(f"F_{n}^{({k})} = {value}")
\end{verbatim}

\end{document}